\documentclass[10pt,reqno]{article}
\usepackage{geometry}
\usepackage{enumerate}
\usepackage{amssymb,amsfonts,amsmath,bbm,mathrsfs,stmaryrd,mathtools,mathabx}
\usepackage{aliascnt}
\usepackage{thmtools}
\usepackage{amsmath,amsthm}
\usepackage{xcolor}
\usepackage{url}
\usepackage{enumerate}
\usepackage{thmtools}
\usepackage[backref=page]{hyperref}
\hypersetup{colorlinks,
linkcolor=black!50!purple,
citecolor=blue
,pdftitle={},
pdfproducer={pdfLaTeX},
pdfpagemode=None,
bookmarksopen=true,
bookmarksnumbered=true}         
\usepackage{tikz}
\usepackage{cleveref}
\creflabelformat{enumi}{#2(#1)#3}
\crefformat{enumi}{(#2#1#3)}
\crefformat{equation}{(#2#1#3)} 
\Crefformat{equation}{(#2#1#3)}

\crefformat{inequality}{(#2#1#3)} 
\Crefformat{Inequality}{(#2#1#3)}

\chardef\bslash=`\\ 

\theoremstyle{plain} 
\newtheorem{theorem}{Theorem}[section]
\theoremstyle{corollary}
\newtheorem{corollary}[theorem]{Corollary}
\theoremstyle{lemma}
\newtheorem{lemma}[theorem]{Lemma}

\theoremstyle{proposition}
\newtheorem{proposition}[theorem]{Proposition}
\theoremstyle{definition}
\newtheorem{definition}[theorem]{Definition}
\newtheorem{example}[theorem]{Example}
\newtheorem{examples}[theorem]{Examples}
\newtheorem{remark}[theorem]{Remark}

\numberwithin{equation}{section}

\newcommand{\Addresses}{{
  \medskip
  \footnotesize

  \textsc{Department of Mathematics, Alabama A\&M University}
  \par\nopagebreak
  \textsc{Huntsville, AL 35811, USA}  
  \par\nopagebreak
  \textit{E-mail address}: \texttt{zhenhua.wang@aamu.edu}}}

\DeclareMathOperator{\Adb}{{\mathbb A}}
\DeclareMathOperator{\Bdb}{{\mathbb B}}

\DeclareMathOperator{\Idb}{{\mathbb I}}

\DeclareMathOperator{\Rdb}{{\mathbb R}}

\DeclareMathOperator{\Al}{{\mathcal A}}

\DeclareMathOperator{\Fl}{{\mathcal F}}
\DeclareMathOperator{\Gl}{{\mathcal G}}
\DeclareMathOperator{\Hl}{{\mathcal H}}

\DeclareMathOperator{\Ul}{{\mathcal U}}

\DeclareMathOperator{\Afk}{{\mathfrak A}}

\newcommand{\Sp}[1]{\operatorname{Sp}(#1)}
\newcommand{\abs}[1]{\left\vert#1\right\vert}

\newcommand{\norm}[1]{\left\Vert#1\right\Vert}

\newcommand{\eval}[2][\right]{\relax
  \ifx#1\right\relax \left.\fi#2#1\rvert}

\let\abs=\envert

\let\norm=\enVert

\begin{document}
\title{\vspace{-10ex}  Lie-Trotter means in JB-algebras\footnote{Part of this work was done at University of Georgia}}
\author{Zhenhua Wang}
\date{\today}
\maketitle
\markboth{Lie-Trotter means in JB-algebras}{Lie-Trotter means in JB-algebras}

\begin{abstract}
  We initiate the study of Lie-Trotter means in JB-algebras, which is an extension of Lie-Trotter formulas in JB-algebras. We show that two-variable Lie-Trotter means include the weighted arithmetic mean, weighted harmonic mean, weighted geometric mean, and weighted spectral geometric mean. Consequently, several generalized Lie-Trotter formulas in JB-algebras are derived. Additionally, we demonstrate that the Sagae-Tanabe mean and Hansen's induction mean in JB-algebras are multivariable Lie-Trotter means. In the end, using arithmetic-geometric-harmonic mean inequalities we provide a characterization of multivariate Lie-Trotter means.
\end{abstract}


\vspace{.5cm}


\tableofcontents

\section{Introduction}
In 1934, Jordan formalized quantum theory based on the structure of Jordan algebras \cite{Jordan34}. These algebras, characterized by their non-associative properties, provide a broader framework for
 investigating quantum mechanics. The initial conjecture was that this quantum formulation could extend its applicability to relativity and nuclear phenomena \cite{Jordan34}. Alternatively, Jordan algebras serve as a conceptual extension of real/complex quantum theory. For instance, in conventional quantum theory, the product of two observables (self-adjoint operators) may not necessarily be an observable. But the Jordan product of two observables is an observable. This research direction is motivated by the fact that observables in a quantum system naturally form a JB-algebra, which is non-associative. Hence, JB-algebras are considered as objects of interest in the study of quantum systems. Beyond this, JB-algebras have found extensive applications in various domains, including analysis, geometry, operator theory, and more. Further details on these applications can be found in \cite{Chu, Upmeier, Upmeier1}.

The Lie-Trotter formula has played a significant  role in quantum mechanics, quantum computing, and quantum simulations. Its significance lies in its ability to decompose a complicated quantum system into simpler systems.  Recently, by utilizing the Jordan product, the study of generalized Lie-Trotter formulas and Suzuki type error estimates in JB-algebras and Banach algebra was initiated by the author and collaborators \cite{Wang2023_LieTrtfml, Wang2023}. Despite an arbitrary number of elements in a non-associative JB-algebra, we demonstrated that Lie-Trotter formulas and Suzuki type estimates still hold for such algebras.

In a different direction, the notion of mean in JB-algebras was introduced by Lawson and Lim \cite{Lawson2007}. Later, unaware of their work, the author and his collaborators independently studied the weighted means and their properties in the setting of JB-algebras \cite{Wang2021_meanJB}. Many identities and inequalities for JB-algebras were established \cite{Wang2021_meanJB}.

This paper extends the Lie-Trotter formulas established in \cite{Wang2023_LieTrtfml} and further explores the study of means in JB-algebras by introducing the concept of Lie-Trotter means in the setting of JB-algebras. \Cref{sct:sptmn} is dedicated to the study of the spectral geometric mean, along with its fundamental properties. In \Cref{sct:2LTtmn}, we show that the weighted arithmetic mean, weighted harmonic mean, weighted geometric mean, and weighted spectral geometric mean are two variable Lie-Trotter means. As a result, several generalized Lie-Trotter formulas in JB-algebras are derived. Furthermore, in \Cref{sct:nLTtmn}, we prove that Sagae-Tanabe mean and Hansen's induction mean in JB-algebras are multivariable Lie-Trotter means. We conclude our paper by presenting a characterization of multivariate Lie-Trotter means through the arithmetic-geometric-harmonic mean inequalities.

\subsection{Some common notation}
We give some background on JB-algebra and fix the notation. For more information, we refer the reader to \cite{Alfsen2003_JBa, Hanche1984_ja, Wang2021_meanJB, Wang2022_entropyJB}.
\medskip

\begin{definition}\label{def:ja}
	A {\bf Jordan algebra} ${\mathfrak A} $ over real number is a vector space ${\mathfrak A}$ over $\Rdb$ equipped with a  bilinear product $\circ$ that satisfies the following identities:
	$$A\circ B =B\circ A, \,\ \,\ (A^2\circ B)\circ A=A^2\circ (B\circ A),$$
where $\displaystyle A^2=A\circ A.$ 
\end{definition}

Any associative algebra ${\mathfrak A}$ has an underlying Jordan algebra structure with Jordan product given by 
	$$A\circ B=(AB+BA)/2.$$
	Jordan subalgebras of such underlying Jordan algebras are called {\bf special}. 
	
\medskip

\begin{definition}\label{def: jba}
	A {\bf JB-algebra} is a Jordan algebra ${\mathfrak A}$ over ${\mathbb R}$ with a complete norm satisfying the following conditions for $A, B\in {\mathfrak A}:$ 
\begin{align*}
\Vert A\circ B\Vert \leq \Vert A\Vert \Vert B\Vert,~~\Vert A^2\Vert=\Vert A\Vert^2,~~{\rm and}~~\Vert A^2\Vert \leq \Vert A^2+B^2\Vert.		
\end{align*}
\end{definition} 
\medskip

As an important object in physics, the set of bounded self adjoint operators on a Hilbert space $H$, denoted by $B(H)_{sa}$, which is the set of observables in a quantum mechanical system, is a JB-algebra. However, it is not an associative algebra.
\medskip

\begin{definition}\label{def:pstvt}
Let $\Afk$ be a unital JB-algebra. 
We say $A\in \Afk$ is {\bf invertible} if there exists $B\in \Afk,$ which is called {\bf Jordan inverse} of $A,$ such that  
\begin{align*}
	A\circ B=I \quad \mbox{and}\quad A^2\circ B=A.	
\end{align*}
An element $S$ of a JB-algebra ${\mathfrak A}$ satisfying $S^2=I$ is called a {\bf symmetry}. Obviously, $S$ is invertible and its inverse is itself. 

 The {\bf spectrum} of $A$ is defined by 
 \begin{align*}
 \Sp{A}:=\{\lambda\in \Rdb\,|\, A-\lambda I\,\ \text{ is not invertible in} \Afk \}.	
 \end{align*}
If $\Sp{A}\subset [0,\infty),$ we say $A$ is {\bf positive}, and write $A \geq 0$. 	
\end{definition}
\medskip

Let ${\mathfrak A}$ be a unital JB-algebra and ${\mathfrak A}_+$ be the closed convex cone of positive elements in ${\mathfrak A}.$ And, we denote ${\mathfrak A}_{++}$ as the open convex cone of positive invertible elements in ${\mathfrak A}.$ 
\medskip

 \begin{definition}\label{def:uoperator}
Let $\Afk$ be a unital JB-algebra and $A, B \in \Afk$.  We define a map $\Ul_A$ on $\Afk$ by
\begin{align*}
\label{JI}
	\Ul_{A}(B):=\{ABA\}:= 2(A\circ B)\circ A -A^2\circ B.	
\end{align*}
 \end{definition}
Note that $ABA$ is meaningless unless $\Afk$ is special, in which case $\{ABA\}=ABA.$ 
The following proposition will be used repeatedly in this paper.
\medskip

\begin{proposition}{\rm \cite[Lemma 1.23-1.25]{Alfsen2003_JBa}} \label{3inv}
Let $\Afk$ be a unital JB-algebra and $A, B\in \Afk$.
\begin{enumerate}[{\rm (1)}]
	\item If $B$ is positive, then $\Ul_A(B)=\{ABA\}\geq 0.$ 
	\item If $A, B$ are invertible, then $\{ABA\}$ is invertible with inverse $\{A^{-1}B^{-1}A^{-1}\}.$
	\item If $A$ is invertible, then $\Ul_A$ has a bounded inverse $\Ul_{A^{-1}}.$

\end{enumerate}	
\end{proposition}
\medskip

The weighed means in JB-algebras were introduced in \cite{Wang2021_meanJB} as follows
\medskip

\begin{definition} \label{def:wtmnJB}
For two positive invertible elements $A, B$ in a unital JB-algebra ${\mathfrak A}$ and $0\leq \lambda\leq 1$, we denote
\begin{eqnarray*}
&&\Al_2(A, B)=A\triangledown_{\lambda} B=(1-\lambda)A+\lambda B;\\
&&\Gl_2(A, B)=A\#_{\lambda} B=\left\{A^{\frac{1}{2}}\{A^{-\frac{1}{2}}BA^{-\frac{1}{2}}\}^{\lambda}A^{\frac{1}{2}}\right\};\\
&&\Hl_2(A, B)= A!_{\lambda} B=\left((1-\lambda)A^{-1}+\lambda B^{-1}\right)^{-1}.
\end{eqnarray*}	
as {\bf weighted arithmetic mean}, {\bf weighted geometric mean} and  {\bf weighted harmonic mean} respectively. When $\lambda=1/2,$ we denote $A\#_{\lambda} B$ as $A\#B.$
\end{definition}

\section{Spectral geometric means}\label{sct:sptmn}

\subsection{Geometric mean revisited}

The following result is the uniqueness of positive invertible solution of generalized Riccati equation in JB-algebras, which is no doubt known to experts, see e.g. \cite[Section 2.5]{Roelands2019}.
\medskip

\begin{proposition}\label{prop:unqgmnJB}
Let $A, B$ be two positive invertible elements in a JB-algebra ${\mathfrak A}$. The Riccati equation $\{XA^{-1}X\}=B$ has a unique positive invertible solution $\displaystyle A\#B$. 	
\end{proposition}
\begin{proof}
By Macdonald Theorem (see e.g. Theorem 1.13 \cite{Alfsen2003_JBa}), a direct computation shows that 
\begin{align*}
\left\{(A\#B)A^{-1}(A\#B)\right\}&=\left\{\left\{A^{\frac{1}{2}}\{A^{-\frac{1}{2}}BA^{-\frac{1}{2}}\}^{\frac{1}{2}}A^{\frac{1}{2}}\right\}A^{-1}\left\{A^{\frac{1}{2}}\{A^{-\frac{1}{2}}BA^{-\frac{1}{2}}\}^{\frac{1}{2}}A^{\frac{1}{2}}\right\}\right\}\\
&=\left\{A^{\frac{1}{2}}\left\{\left\{A^{-\frac{1}{2}}BA^{-\frac{1}{2}}\right\}^{\frac{1}{2}}\{A^{\frac{1}{2}}A^{-1}A^{\frac{1}{2}}\}\left\{A^{-\frac{1}{2}}BA^{-\frac{1}{2}}\right\}^{\frac{1}{2}}\right\}A^{\frac{1}{2}}\right\}\\
&=\left\{A^{\frac{1}{2}}\left\{A^{-\frac{1}{2}}BA^{-\frac{1}{2}}\right\}A^{\frac{1}{2}}\right\}=B.
\end{align*}

Suppose $X,Y$ are two positive solutions to the Riccati equation. Then by \cite[equality (1.15)]{Alfsen2003_JBa}
\begin{align*}
\left\{A^{-1/2}XA^{-1/2}\right\}^2=\left\{A^{-1/2}\left\{XA^{-1}X\right\}A^{-1/2}\right\}=\left\{A^{-1/2}\left\{YA^{-1}Y\right\}A^{-1/2}\right\}=\left\{A^{-1/2}YA^{-1/2}\right\}^2.	
\end{align*}
The uniqueness of positive square roots in JB-algebra indicates that 
\begin{align*}
\left\{A^{-1/2}XA^{-1/2}\right\}=\left\{A^{-1/2}YA^{-1/2}\right\}.	
\end{align*}
Therefore, $\displaystyle X=Y.$
\end{proof}

\medskip

\begin{proposition}\label{prop:smetric}
Let $A, B$ be two positive invertible elements in a unital JB-algebra $\Afk.$ Then the map $d:\Afk_{++}\times \Afk_{++}\to [0,\infty)$ defined by
\begin{align*}
d(A, B)=2\norm{\log(A^{-1}\#B)}	
\end{align*}
has the following properties: 
\begin{enumerate}[{\rm (i)}]
\item \label{item:smmtc} $d$ is a semi-metric.
	\item \label{item:sclm} $\displaystyle d(\alpha A, \alpha B)=d(A, B)$ for any $\alpha>0.$
	\item \label{item:ivs}  $\displaystyle d(A^{-1}, B^{-1})=d(A, B)$
	\item \label{item:untrivrt} $\displaystyle d\left(\{UAU\}, \{UBU\}\right)=d(A, B)$ for any symmetry $U\in \Afk.$
\end{enumerate}	
\end{proposition}

\begin{proof}
For the proof of \Cref{item:smmtc}, the non-negativity of $d(A, B)$ is obvious. If $A=B,$ then $\displaystyle A^{-1}\#B=I.$ Thus, $\displaystyle \log(A^{-1}\# B)=0.$ Conversely, if $d(A, B)=0,$ then $\displaystyle \log(A^{-1}\# B)=0.$ By functional calculus in JB-algebra, $A^{-1}\#B=\exp(0)=I,$ which implies that $\{A^{1/2}BA^{1/2}\}^{1/2}=A.$ Thus, $B=A$ since $B=\{A^{-1/2}A^2A^{-1/2}\}=A.$ By \cite[Collorary 1 and Proposition 6(v)]{Wang2021_meanJB}, $(A^{-1}\# B)^{-1}=A\# B^{-1}=B^{-1}\# A.$ Therefore, 
\begin{align*}
d(A, B)=2\norm{\log(A^{-1}\# B)}	=2\norm{\log(A^{-1}\# B)^{-1}}=2\norm{\log(B^{-1}\# A)}=d(B, A).		
\end{align*}

\cref{item:sclm} follows from the fact $(\alpha A)^{-1}\#(\alpha B)=A^{-1}\#B$ by \cite[Proposition 6(i)]{Wang2021_meanJB}

For the proof of \Cref{item:ivs}, by \cite[Proposition 6(v)]{Wang2021_meanJB}
\begin{align*}
d(A^{-1}, B^{-1})=2\norm{\log(A\# B^{-1})}=2\norm{\log( B\#A^{-1})}=d(B, A)=d(A, B).		\end{align*}

As for \cref{item:untrivrt}, \cite[equation 16]{Wang2021_meanJB} implies that $\log\left(\{U(A^{-1}\# B)U\}\right)=\{U\log\left(A^{-1}\# B\right)U\}.$ Thus,
\begin{align*}
d\big(\{UAU\}, \{UBU\}\big)&=2\norm{\log(\{UAU\}^{-1}\#  \{UBU\})}=2\norm{\log\left(\{U(A^{-1}\# B)U\}\right)}\\
&=2\norm{\{U\log\left(A^{-1}\# B\right)U\}}=2\norm{\log(A^{-1}\# B)}\\
&=d(A, B)
\end{align*}
\end{proof}

\subsection{Spectral Geometric mean}

\begin{definition}\label{def:wspectralgm}
Let $A$ and $B$ be two positive invertible elements in a unital JB-algebra ${\mathfrak A}$ and $\lambda\in \Rdb.$ The {\bf weighted spectral geometric mean} of $A$ and $B$ is defined as
\begin{align*}
A\natural_{\lambda}B:=\left\{\left(A^{-1}\#B\right)^{\lambda}A\left(A^{-1}\#B\right)^{\lambda}\right\}
\end{align*}	
If $\lambda=1/2,$ then we simply write $\displaystyle A\natural_{1/2}B$ as $\displaystyle A\natural B.$
\end{definition}
\medskip

We note that results from \cite[Section 4]{Lee2007_sptrgmtrc} for the weighted spectral geometric mean on symmetric cones, can be generalized to positive invertible elements in a unital JB-algebra.
\medskip

\begin{proposition}\label{prop:unqsptrgm}
Let $\Afk$ be a unital JB-algebra and $A, B\in \Afk_{++}.$ Then $\displaystyle X=A\natural_{\lambda}B$ is the unique positive invertible solution to the equation
\begin{align*}
(A^{-1}\#B)^{\lambda}=A^{-1}\#X	
\end{align*} 	
\end{proposition}
\begin{proof}
If $X=A\natural_{\lambda}B,$ then
\begin{eqnarray}\label{eqt:gmnAX} 
&A^{-1}\#X=\left\{A^{-\frac{1}{2}}\left\{A^{\frac{1}{2}}XA^{\frac{1}{2}}\right\}^{\frac{1}{2}}A^{-\frac{1}{2}}\right\}\quad\quad \nonumber\\
&\quad\quad\quad \quad \quad \quad \quad \quad \quad \quad \quad \quad \quad \quad \quad =\left\{A^{-\frac{1}{2}}\left\{A^{\frac{1}{2}}\left\{\left(A^{-1}\#B\right)^{\lambda}A\left(A^{-1}\#B\right)^{\lambda}\right\}A^{\frac{1}{2}}\right\}^{\frac{1}{2}}A^{-\frac{1}{2}}\right\}\quad\quad 	
\end{eqnarray}

The uniqueness of positive square root in JB-algebras and \cite[Jordan identity (1.15)]{Alfsen2003_JBa} imply 
\begin{align*}
\left\{A^{\frac{1}{2}}\left\{\left(A^{-1}\#B\right)^{\lambda}A\left(A^{-1}\#B\right)^{\lambda}\right\}A^{\frac{1}{2}}\right\}^{\frac{1}{2}}=\left\{A^{\frac{1}{2}}\left(A^{-1}\#B\right)^{\lambda}A^{\frac{1}{2}}\right\}
\end{align*}
Then, the equation (\ref{eqt:gmnAX}) becomes
\begin{align*}
A^{-1}\#X&=	\left\{A^{-\frac{1}{2}} \left\{A^{\frac{1}{2}}\left(A^{-1}\#B\right)^{\lambda}A^{\frac{1}{2}}\right\}A^{-\frac{1}{2}}\right\}\\
&=\left(A^{-1}\#B\right)^{\lambda}.
\end{align*}
If $Y\in {\mathfrak A}$ is another solution, then $\displaystyle A^{-1}\#X=A^{-1}\#Y.$ By \Cref{prop:unqgmnJB},
\begin{align*}
X=\left\{(A^{-1}\#X)A(A^{-1}\#X)\right\}=\left\{(A^{-1}\#Y)A(A^{-1}\#Y)\right\}=Y.
\end{align*}
\end{proof}
\medskip

\begin{remark}
\Cref{prop:unqsptrgm} implies $A\natural_1 B=B.$	
\end{remark}
\medskip

We include some fundamental properties of wighted spectral geometric mean in JB-algebras, which are analogous to those on symmetric cones.
\medskip

\begin{proposition}\label{prop:Spectralgm}
Let $A, B$ be two positive invertible elements in a unital JB-algebra $\Afk$ and $\lambda \in [0, 1].$

\begin{enumerate}[{\rm (i)}]
\item \label{item:spgmjthmgnt} $(\alpha A)\natural_{\lambda}(\beta B)= \alpha^{1-\lambda} \beta^{\lambda} (A\natural_{\lambda}B),$ for any nonnegative numbers $\alpha$ and $\beta.$ 	
\item \label{item:spgmcgrc} $\{C(A\natural_{\lambda}B)C\}=\{CAC\}\natural_{\lambda}\{CBC\},$ for any symmetry $C$ in ${\mathcal A}.$
\item \label{item:spgmsmt} $(A\natural_{\lambda} B)=B\natural_{1-\lambda} A.$
\item \label{item:spgmsfdlt} $(A\natural_{\lambda} B)^{-1}=A^{-1}\natural_{\lambda} B^{-1}.$
\end{enumerate}
\end{proposition}
\begin{proof}
For (\ref{item:spgmjthmgnt}), by\cite[Proposition 6(i)]{Wang2021_meanJB},
\begin{align*}
(\alpha A)\natural_{\lambda}(\beta B)&= \left\{\left[\left(\alpha A\right)^{-1}\#\left(\beta B\right)\right]^{\lambda}(\alpha A)\left[\left(\alpha A\right)^{-1}\#\left(\beta B\right)\right]^{\lambda}\right\}\\
&=\left\{\left[\alpha^{-\frac{\lambda}{2}}\beta^{\frac{\lambda}{2}} \left(A^{-1}\#B\right)^{\lambda}\right](\alpha A)\left[\alpha^{-\frac{\lambda}{2}}\beta^{\frac{\lambda}{2}} \left(A^{-1}\#B\right)^{\lambda}\right]\right\}\\
&=\alpha^{1-\lambda} \beta^{\lambda}\left\{\left(A^{-1}\#B\right)^{\lambda}A\left(A^{-1}\#B\right)^{\lambda}\right\}.
\end{align*}

Proof of (\ref{item:spgmcgrc}). Denote $\displaystyle L=\{CAC\}\natural_{\lambda}\{CBC\}.$ According to \cite[Proposition 6(iv)]{Wang2021_meanJB}
\begin{align*}
L&=\left\{\left[\left\{CAC\right\}^{-1}\#\left\{CBC\right\}\right]^{\lambda}\left(\left\{CAC\right\}\right)\left[\left\{CAC\right\}^{-1}\#\left\{CBC\right\}\right]^{\lambda}\right\}\\
&=\left\{\left\{C(A^{-1}\#B)C\right\}^{\lambda}\left(\left\{CAC\right\}\right)\left\{C(A^{-1}\#B)C\right\}^{\lambda}\right\}
\end{align*}

From \cite[Proposition 2.1]{Rassouli2017_rltopetp}, for $\lambda \in (0, 1)$ and $x \in (0, \infty)$, 
\begin{align*}
x^{\lambda} = \dfrac{\sin(\lambda \pi)}{\pi}\int_0^{\infty} t^{\lambda -1}(1+tx^{-1})^{-1}dt.	
\end{align*}	
By functional calculus in JB-algebras
\begin{align}
\{CAC\}^{\lambda} &= \dfrac{\sin(\lambda \pi)}{\pi}\int_0^{\infty} t^{\lambda -1}\left(1+t\{CAC\}^{-1}\right)^{-1}dt\nonumber\\
&=\dfrac{\sin(\lambda \pi)}{\pi}\int_0^{\infty}t^{\lambda -1}\{C^{-1}(1+tA^{-1})C^{-1}\}^{-1}dt\nonumber\\
&=\dfrac{\sin(\lambda \pi)}{\pi}\int_0^{\infty}t^{\lambda -1}\{C(1+tA^{-1})^{-1}C\}dt\nonumber\\
&=\left\{C \left[\dfrac{\sin(\lambda \pi)}{\pi}\int_0^{\infty}t^{\lambda -1}(1+tA^{-1})^{-1}dt\right] C\right\}\nonumber\\
&=\left\{C A^{\lambda} C\right\} 
\end{align}	
Therefore, 
$$\displaystyle L=\left\{\left\{C(A^{-1}\#B)^{\lambda}C\right\}\left[\left\{CAC\right\}\right]\left\{C(A^{-1}\#B)^{\lambda}C\right\}\right\}.$$ 
By 2.8.6 in \cite{Hanche1984_ja}, 
\begin{align*}
 L&=\left\{C\left\{\left(A^{-1}\#B\right)^{\lambda}A\left(A^{-1}\#B\right)^{\lambda}\right\}C\right\}\\
 &=\{C(A\natural_{\lambda}B)C\}
\end{align*}

For (\ref{item:spgmsmt}), we denote 
\begin{align}
Z&=\left\{\left(B^{-1}\#A\right)^{\lambda-1}\left\{\left(A^{-1}\#B\right)^{\lambda}A \left(A^{-1}\#B\right)^{\lambda}\right\}\left(B^{-1}\#A\right)^{\lambda-1}\right\}\label{eqt:sptrgmniii}
\end{align}
From \cite[Corollary 1 and Proposition 6(iv)]{Wang2021_meanJB}, we know
\begin{align}
Z&=\left\{\left(A^{-1}\# B\right)^{1-\lambda}\left\{\left(A^{-1}\#B\right)^{\lambda}A \left(A^{-1}\#B\right)^{\lambda}\right\}\left(A^{-1}\# B\right)^{1-\lambda}\right\}\nonumber
\end{align}
By Shirshov-Cohen theorem for JB-algebras (see e.g.\, \cite[Theorem 7.2.5]{Hanche1984_ja}),
\begin{align*}
Z=\left\{\left(A^{-1}\#B\right)A\left(A^{-1}\#B\right)\right\}=B.
\end{align*}
Therefore, according to equation \cref{eqt:sptrgmniii}
\begin{align*}
\left\{\left(B^{-1}\#A\right)^{1-\lambda }B\left(B^{-1}\#A\right)^{1-\lambda}\right\}	=\left\{\left(A^{-1}\#B\right)^{\lambda}A \left(A^{-1}\#B\right)^{\lambda}\right\}.
\end{align*}

Proof of (\ref{item:spgmsfdlt}). A same reasoning as in the proof of (iii) implies 
\begin{align*}
(A\natural_{\lambda}B)^{-1}&=\left\{\left(A^{-1}\#B\right)^{\lambda}A\left(A^{-1}\#B\right)^{\lambda}\right\}^{-1}\\ 
&=\left\{\left(A^{-1}\#B\right)^{-\lambda}A^{-1}\left(A^{-1}\#B\right)^{-\lambda}\right\}	\\
&=\left\{\left(A\#B^{-1}\right)^{\lambda}A^{-1}\left(A\#B^{-1}\right)^{\lambda}\right\}\\
&=A^{-1}\natural_{\lambda}B^{-1}.
\end{align*} 
\end{proof}
\medskip

We establish upper and lower bounds for the weighted spectral geometric mean of positive invertible elements in unital JB-algebras. This result finds its roots in \cite[Proposition 2.3]{Kim2015_rltvopetrp}, which pertains to matrices.

\medskip

\begin{proposition}\label{prop:bdwgmn}
Let $A, B$ be two positive invertible elements in ${\mathfrak A}$ and $\lambda\in [0, 1].$ Then
\begin{align*}
2^{1+\lambda}(A+B^{-1})^{-\lambda}-A^{-1}\leq A\natural_{\lambda}B \leq \left[2^{1+\lambda}(A^{-1}+B)^{-\lambda}-A\right]^{-1}.	
\end{align*}
\end{proposition}
\begin{proof}
Let $X=A\natural_{\lambda}B$ for $\lambda \in [0, 1].$ By \Cref{prop:unqsptrgm}, Young's inequalities for JB-algebras (see e.g.\, \cite[Theorem 3]{Wang2021_meanJB}) and \cite[Proposition 5]{Wang2021_meanJB}, we have
\begin{align*}
\left(\frac{A+X^{-1}}{2}\right)^{-1}\leq A^{-1}\# X=(A^{-1}\#B)^{\lambda}\leq \left(\frac{A^{-1}+B}{2}\right)^{\lambda}	
\end{align*}
 The desired result follows from similar argument as in the proof of \cite[Proposition 2.3]{Kim2015_rltvopetrp}.	
\end{proof}
\medskip

\begin{theorem}\label{thm:unqsltfRcteqt} {\rm (Cf.\,\cite[Theorem 3.4]{Gan2023_Spectralgmtmn})}
For $\lambda\in (0, 1),$ let $\displaystyle \Gl_{\lambda}: {\mathfrak A}_{++}\times {\mathfrak A}_{++}\to {\mathfrak A}_{++}$ satisfy $\Gl_{\lambda}(A, A)=A$ for all $A\in {\mathfrak A}_{++}$ and
\begin{align}\label{eqt:unqsltgl}
	\Gl_{\lambda}(A, B)=I\Rightarrow B=A^{1-1/\lambda}\, \mbox{for any}\,\, A, B\in {\mathfrak A}_{++}.
\end{align}	
Then $\displaystyle A\natural_{\lambda}B$ is the unique solution $X\in {\mathfrak A}_{++}$ of the euqation
\begin{align*}
\Gl_{\lambda}(A\# X^{-1}, B\#X^{-1})=I.	
\end{align*}
\end{theorem}
\begin{proof}
Let $U=A\# X^{-1}$ and $V=B\# X^{-1}.$ Then $\displaystyle V=U^{1-1/{\lambda}}$ since $\displaystyle G_{\lambda}(U,V)=I.$  \Cref{prop:unqgmnJB} indicates $\displaystyle X=\{U^{-1}AU^{-1}\}=\{V^{-1}BV^{-1}\}.$ Thus, 
\begin{align*}
A&=\left\{U\left\{V^{-1}BV^{-1}\right\}U\right\}=\{U^{1/\lambda}BU^{1/\lambda}\}
\end{align*}
where the second equation follows from Macdonald Theorem.
Again by \Cref{prop:unqgmnJB}, we know that $\displaystyle U^{1/\lambda}=A\# B^{-1}.$ Therefore, 
\begin{align*}
X=\left\{(A\#B^{-1})^{-\lambda}A(A\#B^{-1})^{-\lambda}\right\}=\left\{(A^{-1}\#B)^{\lambda}A(A^{-1}\#B)^{\lambda}\right\}=A\natural_{\lambda}B	
\end{align*}
where the second equation follows from \cite[Proposition 6(v)]{Wang2021_meanJB}.
\end{proof}

\section{Two-variable Lie-Trotter mean}\label{sct:2LTtmn}
The following result extends \cite[Proposition 1.1]{Ahn2007_LieTrtfml} to the setting of unital JB-algebras. 
\medskip

\begin{proposition}\label{prpst:dfrttrtfml}
For any differentiable curve $\gamma: (-\varepsilon, \varepsilon)\to {\mathfrak A}_{++}$ with $\gamma(0)=I,$
\begin{align*}
e^{\gamma'(0)}=\lim_{t\to 0}\gamma(t)^{1/t}=\lim_{n\to \infty}\gamma(1/n)^{n}.	
\end{align*}	
\end{proposition}
\begin{proof}
The exponential function: $\displaystyle \exp: {\mathfrak A}\to {\mathfrak A}_{++}$ and logarithmic function $\displaystyle \log: {\mathfrak A}_{++}\to {\mathfrak A}$ are well-defined and diffeomorphic. Following a similar argument as in the proof of \cite[Proposition 1.1]{Ahn2007_LieTrtfml}, we have	
\begin{align*}
\gamma'(0)&=(\log\circ \gamma)'(0)=\lim\limits_{t\to 0}\frac{\log(\gamma(t))-\log(\gamma(0))}{t}=\lim\limits_{t\to 0}\frac{\log(\gamma(t))}{t}\\
&=\lim\limits_{t\to 0}\log(\gamma(t)^{1/t})=\lim\limits_{n\to \infty}\log(\gamma(1/n)^{n}).
\end{align*}
Therefore, $\displaystyle \exp\{\gamma'(0)\}=\lim\limits_{t\to 0}\gamma(t)^{1/t}=\lim\limits_{n\to \infty}\gamma(1/n)^{n}.$
\end{proof}

\medskip

\begin{definition}\label{def:wt2mn}
Let $\lambda\in (0, 1).$ A {\bf weighted $2$-mean} in a unital JB-algebra $\Afk$ is  a map 
$$\displaystyle G_2(1-\lambda, \lambda;\,\cdot \,): \Afk_{++}\times \Afk_{++} \to \Afk_{++}$$ satisfying $\displaystyle G_2(1-\lambda, \lambda; A, A)=A,$ for any $A\in {\mathfrak A}_{++}.$
\end{definition}

\medskip

\begin{definition}\label{def:2lietrtmn}
Let $\lambda \in (0, 1).$ A {\bf two-variable Lie-Trotter mean}  in a unital JB-algebra $\Afk$ is a weighted 2-mean $\displaystyle G_2(1-\lambda, \lambda;\,\cdot \,): \Afk_{++}^2\to \Afk_{++}$ such that it is differentiable and satisfies 
\begin{align}\label{eqt:2lietrtfml}
\lim\limits_{t\to 0}G_2(1-\lambda, \lambda;\gamma_1(t), \gamma_2(t))	^{1/t}=\exp[(1-t)\gamma_1'(0)+t\gamma_2'(0)],
\end{align}
where $\displaystyle \gamma_1, \gamma_2:(-\varepsilon, \varepsilon)\to \Afk_{++}$ are differentiable curves with $\displaystyle  \gamma_1(0)=\gamma_2(0)=I.$
\end{definition}

\medskip

\begin{example}\label{exmp:2lietrtmn}
Let $\displaystyle \gamma_1, \gamma_2:(-\varepsilon, \varepsilon)\to {\mathfrak A}_{++}$ be two differentiable curves with $\displaystyle  \gamma_1(0)=\gamma_2(0)=I.$	We denote
\begin{align*}\label{eqt:armhmgmsptrgm}
\Gamma_1(t)&=\Al_2(1-\lambda, \lambda; \gamma_1(t), \gamma_2(t))=(1-\lambda)\gamma_1(t)+\lambda \gamma_2(t)\\	
\Gamma_2(t)&=\Hl_2(1-\lambda, \lambda; \gamma_1(t), \gamma_2(t))=[(1-\lambda)\gamma_1(t)^{-1}+\lambda \gamma_2(t)^{-1}]^{-1}\\
\Gamma_3(t)&=\Gl_2(1-\lambda, \lambda; \gamma_1(t), \gamma_2(t))=\gamma_1(t)\#_{\lambda}\gamma_2(t)\\
\Gamma_4(t)&=\Fl_2(1-\lambda, \lambda; \gamma_1(t), \gamma_2(t))=\gamma_1(t)\natural_{\lambda}\gamma_2(t)
\end{align*}  
We observe that $\Gamma_k$'s are differentiable curves with $\Gamma_k'(0)=(1-\lambda)\gamma_1'(0)+\lambda \gamma_2'(0),$ for any $1\leq k\leq 4.$ By \Cref{prpst:dfrttrtfml},
\begin{align*}
\exp[(1-\lambda)\gamma_1'(0)+\lambda \gamma_2'(0)]&=\lim\limits_{t\to 0}[(1-\lambda)\gamma_1(t)+\lambda \gamma_2(t)]^{1/t}\\
&= \lim\limits_{t\to 0}[(1-\lambda)\gamma_1(t)^{-1}+\lambda \gamma_2(t)^{-1}]^{-1/t}\\
&= \lim\limits_{t\to 0}[\gamma_1(t)\#_{\lambda}\gamma_2(t)]^{1/t}\\
&= \lim\limits_{t\to 0}[\gamma_1(t)\natural_{\lambda}\gamma_2(t)]^{1/t}
\end{align*}
Therefore, the weighted arithmetic mean, weighted harmonic mean, weighted geometric mean, and weighted spectral geometric mean are two-variable Lie-Trotter means in ${\mathfrak A}$. 
\end{example} 

\medskip

The following result presents generalized Lie-Trotter formulas in JB-algebras.

\medskip

\begin{corollary}
Let $A, B$ be two elements in a unital JB-algebra ${\mathfrak A}$. Then
\begin{align*}
\exp[(1-\lambda)A+\lambda B]&=\lim\limits_{t\to 0}\left[(1-\lambda)e^{tA}+\lambda e^{tB}\right ]^{\frac{1}{t}}
=\lim\limits_{n\to \infty}\left[(1-\lambda)e^{\frac{A}{n}}+\lambda e^{\frac{B}{n}}\right ]^{n}\\
&= \lim\limits_{t\to 0}\left[(1-\lambda)e^{-tA}+\lambda e^{-tB}\right]^{-\frac{1}{t}}
=\lim\limits_{n\to \infty}\left[(1-\lambda)e^{-\frac{A}{n}}+\lambda e^{-\frac{B}{n}}\right]^{-n}\\
&= \lim\limits_{t\to 0}\left(e^{tA}\#_{\lambda}e^{tB}\right)^{\frac{1}{t}}
=\lim\limits_{n\to \infty}\left(e^{\frac{A}{n}}\#_{\lambda}e^{\frac{B}{n}}\right)^{n}\\
&= \lim\limits_{t\to 0}\left(e^{tA}\natural_{\lambda}e^{tB}\right)^{1/t}
= \lim\limits_{n\to \infty}\left(e^{\frac{A}{n}}\natural_{\lambda}e^{\frac{B}{n}}\right)^{n}
\end{align*}	
\end{corollary}

\begin{proof}
Let $\displaystyle \gamma_1(t)=e^{tA}$ and $\displaystyle \gamma_2(t)=e^{tB}$ where $A, B\in {\mathfrak A}.$ Then $\displaystyle \gamma_1'(0)=A$ and $\displaystyle \gamma_2'(0)=B.$ Then our results follows directly from \Cref{exmp:2lietrtmn}.
\end{proof}

\section{Multivariate Lie-Trotter mean}\label{sct:nLTtmn}
Let $\omega=(\omega_1,\cdots,\omega_n)\in \Delta_n,$ the simplex of positive probability vectors in $\Rdb^n.$

\medskip

\begin{definition}\label{defn:wgtnmn}
A {\bf weighted $n$-mean	} $G_n$ on ${\mathfrak A}_{++}$ for $n\geq 2$ is a map
 $\displaystyle G_n(\omega;\cdot):{\mathfrak A}_{++}^n\longrightarrow {\mathfrak A}_{++}$ satisfying $\displaystyle G_n(\omega;A,\cdots,A)=A$ for any $A\in {\mathfrak A}_{++}.$ The weighted $n$-mean  $\displaystyle G_n(\omega;\cdot)$ is called a {\bf multivariate Lie-Trotter mean} if it satisfies 
\begin{equation}
\lim_{t\to 0} G_n(\omega;\gamma_1(t),\cdots, \gamma_n(t))^{1/t}=\exp\left(\sum_{k=1}^n\omega_k\gamma_k'(0)\right),	\label{eqt:gnrlzLTrt}
\end{equation}

where $\displaystyle \gamma_{k}:(-\varepsilon, \varepsilon)\to {\mathfrak A}_{++}$ is differentiable with $\gamma_{k}(0)=I,$ for any $1\leq k\leq n.$ 
\end{definition}

For $\Adb=(A_1, A_2,\cdot, A_n)\in {\mathfrak A}_{++}^n,$ $t\in\Rdb$ and any invertible element $C$ in ${\mathfrak A},$ we denote
\begin{align*}
\Adb^t &:=(A_1^t, A_2^t,\cdots,A_n^t)\\
\left\{C\Adb C \right\}&:=	\left(\left\{CA_1C\right\}, \left\{CA_2C\right\},\cdots, \left\{CA_nC\right\}\right).
\end{align*}
\subsection{Sagae-Tanabe mean}
\begin{definition}
Let $\displaystyle \omega=(\omega_1, \omega_2, \cdots, \omega)\in \Delta_n$ and $\Adb=(A_1, A_2,\cdots, A_n)\in {\mathfrak A}_{++}^n.$ Assume that a weighted 2-mean $S_2$ is given. Then the associated {\bf Sagae-Tanabe mean} in JB-algebras is defined inductively as
\begin{align*}
S_n(\omega; \Adb)=S_2(1-\omega_n, \omega_n; S_{n-1}(\hat{\omega};A_1,\cdots, A_{n-1}), A_n)	
\end{align*}
for $n\geq 3,$ where $\displaystyle \hat{\omega}=\frac{1}{1-\omega_n}(\omega_1,\cdots, \omega_{n-1})\in \Delta_{n-1}.$	
\end{definition}

\medskip

\begin{examples}
Let the weighted 2-mean  $S_2$ be the weighted arithmetic mean $\Al_2$ , the weighted harmonic mean $\Hl_2,$ the weighted geometric mean $\Gl_2$ , and the weighted spectral geometric mean $\Fl_2$ respectively. Then by induction, the corresponding Sagae-Tanble means are 
\begin{enumerate}[{\rm(1)}]
\item $\displaystyle \Al_n(\omega; \Adb):=\sum_{k=1}^n \omega_k A_k;$
\item $\displaystyle \Hl_n(\omega; \Adb):=\left(\sum_{k=1}^n \omega_k A_k^{-1}\right)^{-1};$
\item $\displaystyle \Gl_n^S(\omega; \Adb):=\left[\left(A_1\#_{\frac{\omega_2}{\omega_1+\omega_2}}A_2\right)\#_{\frac{\omega_3}{\omega_1+\omega_2+\omega_3}}\cdots\right]\#_{\omega_n}A_n;$
\item $\displaystyle \Fl_n^S(\omega; \Adb):=\left[\left(A_1\natural_{\frac{\omega_2}{\omega_1+\omega_2}}A_2\right)\natural_{\frac{\omega_3}{\omega_1+\omega_2+\omega_3}}\cdots\right]\natural_{\omega_n}A_n.$
\end{enumerate}	
\end{examples}

\medskip

The following result is \cite[Theorem 3.1]{Hwang2017_LTrtmn} in the setting of JB-algebras.
\medskip

\begin{theorem}\label{thm:STmnLTrt}
Let $S_2$ be a Lie-Trotter mean. Then the Sagae-Tanabe mean $S_n$ for $\displaystyle n\geq 2$ is a multivariate Lie-Trotter mean. 	
\end{theorem} 
\begin{proof}
Let $\displaystyle \gamma_1,\cdots, \gamma_n: (-\varepsilon, \varepsilon)\to {\mathfrak A}_{++}$ be differentiable curves with $\gamma_k(0)=I$ for all $1\leq k\leq n.$	By induction, we suppose that $S_{n-1}$ is a multivariate Lie-Trotter mean, then 
\begin{align*}
\lim_{t\to 0}S_{n-1}(\hat{\omega};\gamma_1(t), \gamma_2(t),\cdots, \gamma_{n}(t))^{1/t}=\exp\left(\sum_{k=1}^{n-1}\frac{\omega_k}{1-\omega_n}\gamma_k'(0)\right).	
\end{align*}
Let $\displaystyle \gamma(t):=S_{n-1}\left(\hat{\omega};\gamma_1(t), \gamma_2(t),\cdots, \gamma_{n-1}(t)\right).$ Then by \Cref{prpst:dfrttrtfml}
\begin{align*}
\gamma'(0)=\log \left[\exp\left(\sum_{k=1}^{n-1}\frac{\omega_k}{1-\omega_n}\gamma_k'(0)\right)\right]=	\sum_{k=1}^{n-1}\frac{\omega_k}{1-\omega_n}\gamma_k'(0).
\end{align*}
Since $S_2$ is a Lie-Trotter mean,
\begin{align*}
\lim_{t\to 0}S_{n}\left(\omega;\gamma_1(t), \cdots, \gamma_{n}(t)\right)^{1/t}&=\lim_{t\to 0}S_{2}\left(1-\omega_n, \omega_n; \gamma(t), \gamma_n(t)\right)^{1/t}\\
&=\exp\left((1-\omega_n)\gamma'(0)+\omega_n \gamma_n'(0)\right)\\
&=\exp\left[(1-\omega_n)\left(\sum_{k=1}^{n-1}\frac{\omega_k}{1-\omega_n}\gamma_k'(0)\right)+\omega_n\gamma_n'(0)\right]\\
&=\exp \left(\sum_{k=1}^{n}\omega_k\gamma_k'(0)\right).
\end{align*}
\end{proof}

\medskip

\begin{corollary}\label{crlr:mLTrtmn}
Let $A_k$ be positive invertible elements in a unital JB-algebra $\Afk$ and $t\in \Rdb.$ Then $\displaystyle \Al_n(\omega; \Adb^t), \Hl_n(\omega; \Adb^t), \Gl_n^S(\omega; \Adb^t)$ and $\Fl_n^S(\omega; \Adb^t)$ are multivariate Lie-Trotter means. 
\end{corollary}
\begin{proof}
Let $\displaystyle \gamma_k(t)=A_k^t$ for any $1\leq k\leq n.$ Then $\gamma_k'(0)=\log A_k.$ {\rm \Cref{thm:STmnLTrt}} implies that
\begin{align*}
\exp\left(\sum_{k=1}^n \omega_k \log A_k\right)&=\lim_{t\to 0}\Al_n(\omega; \Adb^t)^{1/t}=\lim_{t\to 0}\Hl_n(\omega; \Adb^t)^{1/t}\\
&=\lim_{t\to 0}\Gl_n^S(\omega; \Adb^t)^{1/t}=\lim_{t\to 0}\Fl_n^S(\omega; \Adb^t)^{1/t}.	
\end{align*}		
\end{proof}

\medskip

The following result is generalized Lie-Trotter formulas for an arbitrary finite number of elements in JB-algebras.
\medskip

\begin{corollary}\label{cor:gLTrf}
Let $\omega=(\omega_1,\cdots,\omega_n)\in \Delta_n.$ For any finite number of elements $A_1, A_2,\cdots, A_n$ in a unital JB-algebra $\Afk,$ 
\begin{align*}
\exp\left(\sum_{k=1}^n \omega_k A_k\right)&=\lim\limits_{t\to 0}\left(\sum_{k=1}^n \omega_k e^{t A_k} \right)^{1/t}\\
&=\lim\limits_{t\to 0} \left(\sum_{k=1}^n \omega_k e^{-tA_k}\right)^{-1/t}\\
&= \lim\limits_{t\to 0}\left\{\left[\left(e^{tA_1}\#_{\frac{\omega_2}{\omega_1+\omega_2}}e^{tA_2}\right)\#_{\frac{\omega_3}{\omega_1+\omega_2+\omega_3}}\cdots\right]\#_{\omega_n}e^{tA_n}\right\}^{1/t}\\
&= \lim\limits_{t\to 0}\left\{\left[\left(e^{tA_1}\natural_{\frac{\omega_2}{\omega_1+\omega_2}}e^{tA_2}\right)\natural_{\frac{\omega_3}{\omega_1+\omega_2+\omega_3}}\cdots\right]\natural_{\omega_n}e^{tA_n}\right\}^{1/t}
\end{align*}	
\end{corollary}

\subsection{Hansen's inductive mean}

\begin{definition}\label{def:Hasenim} 
Given a weighted $2$-mean $H_2$, the associated {\bf Hasen's inductive mean} in JB-algebras is defined by 
\begin{align*}
H_n(\omega; \Adb)=H_{n-1}\left(\hat{\omega}; H_2(1-\omega_n, \omega_n; A_1, A_n),\cdots,H_2(1-\omega_n, \omega_n; A_{n-1}, A_n)\right)	
\end{align*}
for $n\geq 3,$ and where $\displaystyle \hat{\omega}=\frac{1}{1-\omega_n}(\omega_1,\cdots,\omega_{n-1})\in \Delta_{n-1}.$	
\end{definition}

\medskip

\begin{examples}
The weighted arithmetic mean and harmonic mean are two Hansen's inductive means.
\begin{enumerate} [{\rm(1)}]
	\item If $H_2$ is the weighted arithmetic mean $\Al_2,$ then $\displaystyle H_n(\omega; \Adb)=\Al_n(\omega; \Adb)=\sum_{k=1}^n \omega_k A_k.$
	\item If $H_2$ is the weighted arithmetic mean $\Hl_2,$ then $\displaystyle H_n(\omega; \Adb)=\Hl_n(\omega; \Adb)=\left(\sum_{k=1}^n \omega_k A_k^{-1}\right)^{-1}.$
\end{enumerate}
\end{examples}

\medskip

\begin{proposition} \label{prop:nHSgmtcmn}
The Hansen's inductive mean $\Gl^H_n$ induced by the weighted geometric mean $\#_{\lambda}$ in a unital JB-algebra $\Afk$ has the following properties:
\begin{enumerate}[{\rm (1)}]
\item \label{item:jithmg} {\rm (Joint homogeneity)} For positive real numbers 
\begin{align*}
\Gl_n^H(\omega;\alpha_1A_1,\cdots,\alpha_nA_n)=\left(\prod_{k=1}^n \alpha_k^{\omega_k}\right)\Gl_n^H(\omega;\Adb)
\end{align*}

\item \label{item:mntnct} {\rm (Monotonicity)} If $\displaystyle A_k\leq B_k$ for any $1\leq k\leq n,$ then $\displaystyle \Gl_n^H(\omega;\Adb)\leq \Gl_n^H(\omega;\Bdb)$

\item \label{item:idvdccvt} {\rm (Concavity)} $\displaystyle
\Gl_n^H(\omega; A_1,\cdots,A_n)$
 is concave with respect to $A_1, A_2,\cdots, A_n$ individually.
 
\item \label{item: cgrivrc}{\rm (Congruence invariance)} For any invertible element $C$ in JB-algebra,
\begin{align*}
\Gl_n^H\left(\omega;\{C\Adb C\}\right)=\left\{C\Gl_n^H(\omega;\Adb)C\right\}	
\end{align*}

\item \label{item:slfdalt} {\rm (Self duality)} $\displaystyle \Gl_n^H\left(\omega; \Adb^{-1}\right)^{-1}=\Gl_n^H\left(\omega;\Adb\right)$
\end{enumerate}
\end{proposition}
\begin{proof}
We will prove \cref{item:jithmg} by induction. For $n=2,$ it is \cite[Proposition 6(i)]{Wang2021_meanJB}. For $n=3,$
\begin{align*}
\Gl_3^H(\omega; A_1, A_2, A_3)&=\Gl_2^H(\hat{\omega}; A_1\#_{\omega_3}A_3	, A_2\#_{\omega_3}A_3)=(A_1\#_{\omega_3}A_3)	\#_{\frac{\omega_2}{\omega_1+\omega_2}}(A_2\#_{\omega_3}A_3)
\end{align*}
Thus,
\begin{align*}
\Gl_3^H(\omega; \alpha_1 A_1, \alpha_2 A_2, \alpha_3 A_3)&=\left[(\alpha_1 A_1)\#_{\omega_3}(\alpha_3 A_3)\right]	\#_{\frac{\omega_2}{\omega_1+\omega_2}}\left[(\alpha_2 A_2)\#_{\omega_3}(\alpha_3A_3)\right]\\
&=\left[\alpha_1^{1-\omega_3}\alpha_3^{\omega_3}(A_1\#_{\omega_3}A_3)\right]	\#_{\frac{\omega_2}{\omega_1+\omega_2}}\left[\alpha_2^{1-\omega_3}\alpha_3^{\omega_3}(A_2\#_{\omega_3}A_3)\right]\\
&=\left(\alpha_1^{1-\omega_3}\alpha_3^{\omega_3}\right)^{\frac{\omega_1}{\omega_1+\omega_2}}\left(\alpha_2^{1-\omega_3}\alpha_3^{\omega_3}\right)^{\frac{\omega_2}{\omega_1+\omega_2}}(A_1\#_{\omega_3}A_3)	\#_{\frac{\omega_2}{\omega_1+\omega_2}}(A_2\#_{\omega_3}A_3)\\
&=\alpha_1^{\omega_1}\alpha_2^{\omega_2}\alpha_3^{\omega_3}(A_1\#_{\omega_3}A_3)	\#_{\frac{\omega_2}{\omega_1+\omega_2}}(A_2\#_{\omega_3}A_3)\\
&=\alpha_1^{\omega_1}\alpha_2^{\omega_2}\alpha_3^{\omega_3}	\Gl_3^H(\omega; A_1, A_2, A_3)
\end{align*}

Suppose joint homogeneity holds for any $2\leq m\leq n-1.$ For $m=n,$
\begin{align*}
\Gl_n^H(\omega; \alpha_1A_1,\cdots,\alpha_nA_n)&=\Gl_{n-1}^H\left(\hat{\omega}; (\alpha_1 A_1)\#_{\omega_n}(\alpha_n A_n),\cdots,(\alpha_{n-1} A_{n-1})\#_{\omega_n}(\alpha_n A_n)\right)\\
&=\Gl_{n-1}^H\left(\hat{\omega}; \alpha_1^{1-\omega_n} \alpha_n^{\omega_n}(A_1\#_{\omega_n} A_n),\cdots,\alpha_{n-1}^{1-\omega_n} \alpha_n^{\omega_n}(A_{n-1}\#_{\omega_n} A_n)\right)\\
&=\left(\prod_{k=1}^{n-1} \alpha_k^{1-\omega_n}\alpha_n^{\omega_n}\right)^{\frac{\omega_k}{1-\omega_n}}\Gl_{n-1}^H\left(\hat{\omega}; A_1\#_{\omega_n}A_n,\cdots, A_{n-1}\#_{\omega_n}A_n\right)\\
&=\left(\prod_{k=1}^n \alpha_k^{\omega_k}\right)\Gl_n^H(\omega;\Adb)	
\end{align*}

By induction, \cite[Proposition 6(ii), (iii), (iv), (v)]{Wang2021_meanJB} imply \cref{item:mntnct}, \cref{item:idvdccvt}, \cref{item: cgrivrc} and \cref{item:slfdalt} respectively.
\end{proof}

\medskip

The next Theorem generalizes \cite[Theorem 3]{Wang2021_meanJB} for an arbitrary finite number of positive invertible elements in a unital JB-algebra $\Afk$.

\medskip

\begin{theorem}\label{thm:gYieqlt}{\rm (Generalized Young inequalities for JB-algebras)}
Let $\displaystyle \omega=(\omega_1, \omega_2, \cdots, \omega)\in \Delta_n$ and $\Adb=(A_1, A_2,\cdots, A_n)\in {\mathfrak A}_{++}^n.$
\begin{align} \label{eqt:gYie}
\left(\sum_{k=1}^n \omega_k A_k^{-1}\right)^{-1}\leq \Gl_n^H\left(\omega,\Adb\right)\leq \sum_{k=1}^n \omega_k A_k.	
\end{align}
	
\end{theorem}
\begin{proof}
For $n=2,$ it is \cite[Theorem 3]{Wang2021_meanJB}. For $n=3,$  by \cite[Theorem 3]{Wang2021_meanJB}
\begin{align*}
\Gl_3^H(\omega; A_1, A_2, A_3)&=(A_1\#_{\omega_3}A_3)	\#_{\frac{\omega_2}{\omega_1+\omega_2}}(A_2\#_{\omega_3}A_3)\\
&\leq \frac{\omega_1}{\omega_1+\omega_2}\left[(1-\omega_3)A_1+\omega_3 A_3\right]+\frac{\omega_2}{\omega_1+\omega_2}\left[(1-\omega_3)A_2+\omega_3 A_3\right]\\
&=  \omega_1A_1+\omega_2A_2+\omega_3A_3
\end{align*}
By \Cref{prop:nHSgmtcmn} \cref{item:slfdalt},
\begin{align*}
\Gl_3^H(\omega; A_1, A_2, A_3)=	\Gl_3^H(\omega; A_1^{-1}, A_2^{-1}, A_3^{-1})^{-1}\geq \left(\omega_1A_1^{-1}+\omega_2A_2^{-1}+\omega_3A_3^{-1}\right)^{-1}.
\end{align*}
Thus, \Cref{eqt:gYie} hold for $n=3.$
Suppose \Cref{eqt:gYie} hold for any $2\leq m\leq n-1.$ For $m=n,$
by induction,
\begin{align}\label{eqt:Hsgmatmtgmt}
\Gl_n^H(\omega; \Adb)&=\Gl_{n-1}^H(\hat{\omega}; A_1\#_{\omega_n}A_n,\cdots,  A_{n-1}\#_{\omega_n}A_n)\nonumber\\
&\leq \sum_{k=1}^{n-1} \frac{\omega_k}{1-\omega_n}\left[(1-\omega_n)A_k+\omega_n A_n\right]\nonumber\\
&=\sum_{k=1}^n \omega_kA_k.	
\end{align}
Applying \Cref{prop:nHSgmtcmn} \cref{item:slfdalt} to \Cref{eqt:Hsgmatmtgmt}, we know that
\begin{align*}
\Gl_n^H(\omega; \Adb)\geq \left(\sum_{k=1}^n \omega_kA_k^{-1}\right)^{-1}.	
\end{align*}
\end{proof}
\medskip

The following result extends Theorem 3.5 from \cite{Hwang2017_LTrtmn} to the setting of JB-algebras.

\medskip

\begin{theorem}\label{thm:HsmLTrtm}
Suppose a weighted 2-mean $H_2$ is a Lie-Trotter mean. Then the associated Hansen's inductive mean $H_n$ for $n\geq 2$ is a multivariate Lie-Trotter mean.	
\end{theorem}
\begin{proof}
Let $\omega=(\omega_1,\cdots,\omega_{n-1},\omega_n)\in \Delta_n$ and let $\gamma_1,\cdots,\gamma_n:(-\varepsilon, \varepsilon)\to \Afk_{++}$ be any differentiable curves with $\gamma_k(0)=I$ for any $1\leq k\leq n.$ We need to show that 
\begin{align*}
\lim\limits_{t\to 0}H_n(\omega; \gamma_1(t),\cdots,\gamma_n(t))^{1/t}=\exp\left[\sum_{k=1}^n \omega_k \gamma_k'(0)\right]	
\end{align*}

We will use induction to prove the statement. It is obvious for $n=2$ by our assumption. Suppose that the claim is true for $n-1.$ Then
\begin{align*}
H_n(\omega; \gamma_1(t),\cdots,\gamma_n(t))=H_{n-1}(\hat{\omega}; \beta_1(t),\cdots,\beta_{n-1}(t)),	
\end{align*}
where 
$$\beta_k(t)=H_2(1-\omega_n, \omega_n; \gamma_k(t),\gamma_n(t))\,\, \mbox{for}\, 1\leq k\leq n-1.$$
We observe that $\beta_k(t)$ is a differentiable curve and $\beta_k(0)=I$ by idempotence of $H_2.$ 
In addition, by \Cref{prpst:dfrttrtfml},
\begin{align*}
\beta_k'(0)=(\log\circ \beta_k)'(0)=\lim\limits_{t\to 0}\frac{\log(\beta_k(t))}{t}=\lim\limits_{t\to 0}\log(\beta_k(t))^{1/t}=(1-\omega_n)\gamma_k'(0)+\omega_n \gamma'_n(0).	
\end{align*}
Therefore 
\begin{align*}
\lim\limits_{t\to 0}H_n(\omega; \gamma_1(t),\cdots,\gamma_n(t))^{1/t}&=\lim\limits_{t\to 0}H_{n-1}(\hat{\omega}; \beta_1(t),\cdots,\beta_{n-1}(t))^{1/t}=\exp\left(\sum_{k=1}^{n-1} \frac{\omega_k}{1-\omega_n}\beta_k'(0)\right)\\
&=\exp\left\{\sum_{k=1}^{n-1} \frac{\omega_k}{1-\omega_n}\left[(1-\omega_n)\gamma_k'(0)+\omega_n \gamma'_n(0)\right]\right\}\\
&=\exp\left(\sum_{k=1}^{n} \omega_k \gamma_k'(0)\right)	
\end{align*}
\end{proof}
\section{A characterization of multivariate Lie-Trotter means}\label{sct:cLTtmn}
In this section, we give a characterization of multivariate Lie-Trotter means in a unital JB-algebra $\Afk$.

Let $G_n^{\omega}:=G_n(\omega;\cdot): \Afk_{++}^n\to \Afk_{++}$	be a weighted $n$-mean satisfying the following inequality 
\begin{equation}\label{eqt:hwa}
\Hl_n\leq G_n^{\omega}\leq \Al_n, 	
\end{equation}
where $\Hl_n$ and $\Al_n$ are harmonic and arithmetic means correspondingly. 

Let $A_1,\cdots, A_n\in \Afk.$ Without loss of generality, we suppose there exists at least a nonzero $A_k.$ Denote 
\begin{align*}
\rho:=\max\limits_{1\leq k\leq n}\sigma(A_k),	
\end{align*}
where  $\sigma(A_k)$ represents the spectral radius of $A_k.$ Then $\rho>0.$ For any $t\in (-1/\rho, 1/\rho),$ let $\lambda\in \Sp{I+tA_k}.$ By spectral theory of JB-algebra
\begin{align*}
\lambda>1-\abs{t}\cdot\rho>0	.
\end{align*}
Therefore, 
$G_n^{\omega}(I+tA_1,\cdots,I+tA_n)$ is well defined on the interval $(-1/\rho, 1/\rho).$
 
\medskip

\begin{lemma}
If $G_n^{\omega}$ as defined above satisfies the inequality $\Cref{eqt:hwa},$
then $G_n^{\omega}$ is differentiable at $\Idb=(I,\cdots, I)$ in the sense that
\begin{align*}
\lim\limits_{t\to 0}	\frac{G_n^{\omega}(I+tA_1,\cdots,I+tA_n)-G_n^{\omega}(I,\cdots,I)}{t}
\end{align*}
exists. Moreover, if we denote the limit as $DG_n^{\omega}(\Idb)(A_1,\cdots, A_n),$ then 
$$DG_n^{\omega}(\Idb)(A_1,\cdots, A_n)=\sum_{k=1}^n\omega_k A_k.$$
\end{lemma}
\begin{proof}
Define $\displaystyle f(t)=\left[\sum_{k=1}^n\omega_k (I+tA_k)^{-1}\right]^{-1}$ on $(-1/\rho, 1/\rho)$.	
 Then $f(t): (-1/\rho, 1/\rho)\to \Afk_{++} $ is well-defined with $f(0)=I$ and differentiable at $t=0.$  Indeed,
\begin{align*}
f'(0)&=\lim_{t\to 0}	\frac{f(t)-f(0)}{t}
=\lim_{t\to 0}\frac{\left[\sum_{k=1}^n\omega_k (I+tA_k)^{-1}\right]^{-1}-I}{t}\\
&=\lim_{t\to 0}\frac{\left[\sum_{k=1}^n\omega_k (I+tA_k)^{-1}\right]^{-1}\circ\left[I-\sum_{k=1}^n\omega_k (I+tA_k)^{-1}\right] }{t}\\
&=\lim_{t\to 0}\frac{I-\sum_{k=1}^n\omega_k (I+tA_k)^{-1}}{t}
=\lim_{t\to 0}\sum_{k=1}^n\omega_k\frac{I-(I+tA_k)^{-1}}{t}\\
&=\lim_{t\to 0}\sum_{k=1}^n\omega_k\frac{(I+tA_k)^{-1}\circ(I+tA_k-I)}{t}
=\sum_{k=1}^n\omega_k A_k.
\end{align*}
If $t\in (0,1/\rho)$	and $t\to 0^+,$ then by the inequality \Cref{eqt:hwa}
\begin{align*}
\lim_{t\to 0^+}\frac{f(t)-I}{t}\leq \lim_{t\to 0^+}\frac{G_n^{\omega}(I+tA_1,\cdots,I+tA_n)-I}{t}\leq \lim_{t\to 0^+}\frac{\sum_{k=1}^n\omega_k(I+tA_k)-I}{t}	
\end{align*}
Thus, by sandwich principle
\begin{align*}
\lim_{t\to 0^+}\frac{G_n^{\omega}(I+tA_1,\cdots,I+tA_n)-G_n^{\omega}(I,\cdots,I)}{t}=\sum_{k=1}^n\omega_k A_k	
\end{align*}
Similarly,
\begin{align*}
\lim_{t\to 0^-}\frac{G_n^{\omega}(I+tA_1,\cdots,I+tA_n)-G_n^{\omega}(I,\cdots,I)}{t}=\sum_{k=1}^n\omega_k A_k	
\end{align*}
Therefore, $G_n^{\omega}$ is differentiable at $\Idb$ with $DG_n^{\omega}(A_1,\cdots,A_n)=\sum_{k=1}^n\omega_iA_k.$
\end{proof}

\medskip

\begin{theorem}
If a weighted $n$-mean $G_n$ in a unital JB-algebra $\Afk$ satisfies the inequality $\Cref{eqt:hwa}$, then $G_n$ is the multivariate Lie-Trotter mean.	
\end{theorem}

\begin{proof}
Let $\omega=(\omega_1,\cdots,\omega_{n-1},\omega_n)\in \Delta_n$ and let $\gamma_1,\cdots,\gamma_n:(-\varepsilon, \varepsilon)\to \Afk_{++}$ be any differentiable curves with $\gamma_k(0)=I$ for any $1\leq k\leq n.$ Then the inequality $\Cref{eqt:hwa}$ indicates that
\begin{align*}
	\left[\sum_{k=1}^n\omega_k \gamma_k(t)\right]^{-1}\leq G_n(\omega;\gamma_1(t),\cdots,\gamma_n(t))\leq \sum_{k=1}^n \omega_k\gamma_k(t).
\end{align*}
Since logarithmic function is operator monotone in JB-algebra by \cite[Proposition 5]{Wang2021_meanJB}, 
\begin{align*}
\log \left[\sum_{k=1}^n\omega_k \gamma_k(t)\right]^{-1}\leq \log \left[G_n(\omega;\gamma_1(t),\cdots,\gamma_n(t))\right]\leq \log\left[\sum_{k=1}^n \omega_k\gamma_k(t)	\right]
\end{align*}
For any $t\in (0, \varepsilon)$ and $t\to 0^+$,	 
\begin{align*}
\lim_{t\to 0^+}\log \left[\sum_{k=1}^n\omega_k \gamma_k(t)\right]^{-1/t}\leq\lim_{t\to 0^+} \log \left[G_n(\omega;\gamma_1(t),\cdots,\gamma_n(t))\right]^{1/t}\leq \lim_{t\to 0^+}\log\left[\sum_{k=1}^n \omega_k\gamma_k(t)	\right]^{1/t}
\end{align*}

Using the fact $\Al_n$ and $\Hl_n$ are multivariate Lie-Trotter means by \Cref{crlr:mLTrtmn}, we conclude 
$$\lim_{t\to 0^+} \log \left[G_n(\omega;\gamma_1(t),\cdots,\gamma_n(t))\right]^{1/t}=\sum_{k=1}^n \omega_k\gamma_k'(0).$$
Note that the map $\log: \Afk_{++}\to \Al$ is diffeomorphic, 
$$\lim_{t\to 0^+} \left[G_n(\omega;\gamma_1(t),\cdots,\gamma_n(t))\right]^{1/t}=\exp\left[\sum_{k=1}^n \omega_k\gamma_k'(0)\right].$$

Following a similar argument, we obtain 
\begin{align*}
\lim_{t\to 0^-} \left[G_n(\omega;\gamma_1(t),\cdots,\gamma_n(t))\right]^{1/t}=\exp\left[\sum_{k=1}^n \omega_k\gamma_k'(0)\right].	
\end{align*}
\end{proof}

\bigskip
\noindent {\bf Acknowledgement:}\  We extend our gratitude to Shuzhou Wang for his careful reading of the manuscript and valuable comments.


\addcontentsline{toc}{section}{References}

\def\polhk#1{\setbox0=\hbox{#1}{\ooalign{\hidewidth
  \lower1.5ex\hbox{`}\hidewidth\crcr\unhbox0}}}
  \def\polhk#1{\setbox0=\hbox{#1}{\ooalign{\hidewidth
  \lower1.5ex\hbox{`}\hidewidth\crcr\unhbox0}}}
  \def\polhk#1{\setbox0=\hbox{#1}{\ooalign{\hidewidth
  \lower1.5ex\hbox{`}\hidewidth\crcr\unhbox0}}}
  \def\polhk#1{\setbox0=\hbox{#1}{\ooalign{\hidewidth
  \lower1.5ex\hbox{`}\hidewidth\crcr\unhbox0}}}
  \def\polhk#1{\setbox0=\hbox{#1}{\ooalign{\hidewidth
  \lower1.5ex\hbox{`}\hidewidth\crcr\unhbox0}}}

\Addresses

\end{document}